\documentclass[sn-mathphys-num]{sn-jnl}

\usepackage{graphicx}%
\usepackage{multirow}%
\usepackage{amsmath,amssymb,amsfonts}%
\usepackage{amsthm}%
\usepackage{mathrsfs}%
\usepackage[title]{appendix}%
\usepackage{xcolor}%
\usepackage{textcomp}%
\usepackage{manyfoot}%
\usepackage{booktabs}%
\usepackage{algorithm}%
\usepackage{algorithmicx}%
\usepackage{algpseudocode}%
\usepackage{listings}%
\usepackage{bigints}

\theoremstyle{thmstyleone}%
\newtheorem{theorem}{Theorem}
\newtheorem{proposition}[theorem]{Proposition}%

\theoremstyle{thmstyletwo}%

\theoremstyle{thmstylethree}%

\renewcommand{\theequation}{\thesection.\arabic{equation}}
\newcommand{\norm}[1]{\left\lVert #1\right\rVert}

\newtheorem{lemma}{\bf Lemma}[section]

\begin{document}

\title[Article Title]{A note on improvement by iteration for the approximate solutions of second kind Fredholm integral equations with Green's kernels}

\author*[1]{\fnm{Gobinda} \sur{Rakshit}}\email{g.rakshit@rgipt.ac.in}

\author[2]{\fnm{Shashank K.} \sur{Shukla}}\email{shashankks@rgipt.ac.in}

\author[3]{\fnm{Akshay S.} \sur{Rane}}\email{akshayrane11@gmail.com}

\affil[1,2]{\orgdiv{Department of Mathematical Sciences}, \orgname{Rajiv Gandhi Institute of Petroleum Technology}, \orgaddress{\street{Jais}, \city{Amethi}, \postcode{229304}, \state{Uttar Pradesh}, \country{India}}}


\affil[3]{\orgdiv{Department of Mathematics}, \orgname{Institute of Chemical Technology}, \orgaddress{\street{Nathalal Parekh Marg, Matunga}, \city{Mumbai}, \postcode{400019}, \state{Maharashtra}, \country{India}}}

\abstract{Consider a linear operator equation $x - Kx = f$, where $f$ is given and $K$ is a Fredholm integral operator with a Green's function type kernel defined on $C[0, 1]$. For $r \geq 0$, we employ the interpolatory projection at $2r + 1$ collocation points (not necessarily Gauss points) onto a space of piecewise polynomials of degree $\leq 2r$ with respect to a uniform partition of $[0, 1]$.
	Previous researchers have established that, in the case of smooth kernels with piecewise polynomials of even degree, iteration in the collocation method and its variants improves the order of convergence by projection methods. In this article, we demonstrate the improvement in order of convergence by modified collocation method when the kernel is of Green's function type.}



\maketitle

\section{Introduction}

Let $X = C[0, 1]$, and $\Omega = [0, 1] \times [0, 1]$. Let $\kappa : \Omega \rightarrow \mathbb{R}$ be a real-valued continuous function. Consider the following Fredholm integral operator
\begin{equation} \label {Eq: 1}
	(Kx)(s) = \int_{0}^{1} \kappa(s, t) x(t)~dt, \quad x \in X, \quad s \in [0, 1]. 
\end{equation}
Then $K : C[0, 1] \rightarrow C[0, 1]$ is a compact linear operator.
Consider the following Fredholm integral equation
\begin{equation} \label {Eq: 2}
	x - Kx = f, 
\end{equation}
where $f \in X$ is given. Assume that $1$ is not an eigenvalue of $K$, that is $I - K$ is continuously invertible. So, the equation \eqref{Eq: 2} possesses a unique solution (say) $\varphi$.\\
As it is challenging to find an exact solution of \eqref{Eq: 2}, our focus turns towards finding approximate solution of this equation. One approach involves collocation methods in which we choose a sequence of continuous finite rank interpolatory projections that point-wise converges to the identity operator. These projections are used to define operators that approximate the operator $K$.  

Let $X_n$ be a sequence of finite dimensional approximating subspaces of $X$ and let $P_n$ be a sequence of interpolatory projections from $X$ to $X_n$. There are various techniques for obtaining approximate solutions of \eqref{Eq: 2} based on interpolation discussed in \cite{KEA0}, \cite{Cha1}, \cite{RPK1}, \cite{Qun}. In the collocation method, the equation \eqref{Eq: 2} is approximated by
\begin{equation*} 
	\phi_{n}^C - P_nK\phi_{n}^C = P_nf.
\end{equation*}
and the iterated collocation solution is defined by
\begin{equation*} 
	\phi_{n}^S = K\phi_{n}^C + f.
\end{equation*}
The above iterated collocation method was introduced by Sloan \cite{Sl1}.\\
In Kulkarni \cite{RPK1}, the following modified collocation method of approximation has been introduced:
\begin{equation*} 
	\phi_{n}^M - K_{n}^M\phi_{n}^M = f,
\end{equation*}
where
\[	K_n^M = P_nK + KP_n - P_nKP_n.\]
The iterated modified collocation solution is defined as
\begin{equation*} 
	\tilde{\phi}_n^M=K\phi_n^M+f. 
\end{equation*} 
Let a uniform partition of $[0,1]$ be defined as $\{0 = t_0 < t_1 < \cdots < t_n = 1\}$, where $t_j = \frac{j}{n}$ for $j = 1,2, \ldots, n$. Let $h = \frac{1}{n}$. For $r \geq 0$, let $X_n$ be the space of piecewise polynomials of degree $ \le 2r$ associated with the above partition.
Choose $2r +1$ distinct points in each subinterval as collocation points, and let $P_n: X \rightarrow X_n$ be the interpolatory projection at the collocation points. Then, $P_nx|_{[t_{j-1}, t_j]}$ is a polynomial of degree $\leq 2r$ for each $j = 1, 2, \ldots, n$. 
If $f \in C^{2r +1}([0, 1])$, then $\phi \in C^{2r +1}([0, 1])$. (See \cite{Cha-Leb1}). If $\kappa \in C^{2r+1}(\Omega)$, then from \cite[Theorem 3.1]{RPK1}, we have
\begin{equation*} 
	\norm{\phi - \phi_{n}^C}_\infty = O\left(h^{2r +1}\right), \quad \norm{\phi - \phi_{n}^M}_\infty = \norm{\phi - \tilde{\phi}_n^M}_\infty = O\left(h^{4r + 2}\right).
\end{equation*}
Here, one step of iteration does not improve the order of convergence. Now, divide each subinterval $[t_{j-1}, t_j]$ by $2r$ parts and choose the nodes as the $2r+1$ collocation points. Then we have the following results in the case of sufficiently smooth kernel,
\begin{equation*}
	\norm{\phi - \phi_{n}^C}_\infty = O\left(h^{2r +1}\right),  \quad \norm{\phi - \phi_{n}^S}_\infty = O\left(h^{2r +2}\right).
\end{equation*}
(See Atkinson \cite[Section 3.4.3]{KEA0}.)
In this case, the modified and iterated modified projection has been considered in Rane \cite{RANE}, which gives the following orders of convergence. If $\kappa \in C^{4r + 2}(\Omega)$ and $\phi \in C^{4r}([0, 1])$, then
\begin{equation*}
	\norm{\phi - \phi_n^M }_\infty = O(h^{4r + 3}), \quad \norm{\phi - \tilde{\phi}_n^M }_\infty = O(h^{4r + 4}).
\end{equation*}
However, in the above cases, the iterated solutions (i.e. $\phi_{n}^S$ and $\tilde{\phi}_n^M$) converge to $\phi$ faster than $\phi_{n}^C$ and $\phi_{n}^M$, respectively.

In this paper, our aim is to obtain the orders of convergence for the iterated collocation solution ($\phi_{n}^S$), modified collocation solution ($\phi_n^M$) and its iterated version $\tilde{\phi}_n^M$ in the case of the Fredholm integral operator with the Green's function type kernel using the same collocation points as mentioned above. 

The paper is organized as follows. In Section 2, we establish the framework for our analysis by defining the approximating space, interpolatory projection, and Green's function type kernels. We also introduce some important notations in this section. Towards the end of this section, we prove two crucial lemmas based on the divided difference of $Kx$. Section 3 is dedicated to establishing some crucial estimates which will be used to obtain our main results. Finally, the conclusions from our study are drawn in section 4.
\section{Preliminaries}
Let $X=C[0,1]$ be the space of all real valued continuous functions on $[0,1]$. The Fredholm integral operator $K: X \to X$ is given as 
$$
(Kx)(s) = \int_{0}^{1} \kappa(s, t) x(t)~dt,  \quad x \in X, \quad s \in [0, 1],
$$
where the kernel $\kappa(s, t)$ is a real-valued continuous function of Green's function type.
Note that, for $r \ge 1$, it is assumed that $\alpha > r$, and the kernel $\kappa$ has the following properties:\\
Let 
$$
\Omega_1= \{(s,t) : 0\leq t\leq s\leq 1\}, \quad \Omega_2= \{(s,t) : 0\leq s\leq t\leq 1\}.
$$
There exist functions $\kappa_i \in C^\alpha(\Omega_i), \; i=1,2$, with 
\begin{equation*}
	\kappa(s,t) ~= ~
	\begin{cases}
		\kappa_1(s,t)    ~~~          & (s,t) \in \Omega_1,\\
		\kappa_2(s,t)    ~~~          & (s,t) \in \Omega_2,\\
	\end{cases} 
\end{equation*} 
and $\kappa_1(s,s)=\kappa_2(s,s).$ Such a kernel $\kappa$ is referred as Green's function type of kernel. Clearly, $\kappa$ is a non-smooth kernel. We use the following notations throughout the article:\\
Let $j, k$ be non-negative integer and we set
$$
D^{(j,k)}\kappa_i(s,t)=\frac{\partial^{j+k}\kappa_i}{\partial s^j \partial t ^k} (s,t), \quad \text{for} \; i=1,2.
$$
Define
\begin{equation*}
	C_1= \max_{j, k =0,1,2}\bigg \{ \sup_{ (s,t) \in \Omega_i} \left| D^{(j,k)} \kappa_i(s,t)\right| : \; \text{for}\; i=1, 2\bigg \},
\end{equation*}  
For  $l=1,2,$ we define
\begin{equation*}
	M_l = \max_{1 \le k \le n}\bigg \{ \sup_{t_{k-1} \le t \le s \le t_k} \left| D^{(l,0)} \kappa_1(s,t)\right|, \sup_{t_{k-1} \le  s \le t \le t_k} \left| D^{(l,0)} \kappa_2(s,t)\right| \bigg \}.
\end{equation*} 
For a positive integer $\alpha$, if $x \in C^{2r+\alpha}([0,1])$, we define
\begin{equation*}
	\norm{x}_{2r + \alpha, \infty} = \max_{0 \le i \le 2r +\alpha} \norm{x^{(i)}}_\infty = \max_{0 \le i \le 2r +\alpha} \left\{ \sup_{t \in [0,1]}  \left|x^{(i)}(t)\right| \right\}.
\end{equation*}   
where $x^{(i)}$ is the $i^{th}$ derivative of the function $x$, and
$$
\norm{x^{(i)}}_\infty =  \sup_{t \in [0,1]}  \left|x^{(i)}(t)\right|.
$$
In particular, 
$$
\norm{x}_\infty =  \sup_{t \in [0,1]}  \left|x(t)\right|.
$$
Now, for $x \in X$ and $s \in [0,1]$ we see that
$$
(Kx)(s) = \int_{0}^{1} \kappa(s, t) x(t)~dt= \int_{0}^{s} \kappa_1(s, t) x(t)~dt + \int_{s}^{1} \kappa_2(s, t) x(t)~dt.
$$
Using Leibnitz rule
\begin{align*}
	(Kx)'(s) &= \int_{0}^{s} \frac{\partial \kappa_1}{\partial s}(s, t) x(t)~dt + \int_{s}^{1} \frac{\partial \kappa_2}{\partial s}(s, t) x(t)~dt,
\end{align*}
which implies
$$
| (Kx)'(s) | \le \left( \sup_{0 \le t \le s \le 1} \left| D^{(1,0)} \kappa_1(s,t) \right|  \right) s \norm{x}_\infty +  \left( \sup_{0 \le s \le t \le 1} \left| D^{(1,0)} \kappa_2(s,t) \right|  \right) (1-s) \norm{x}_\infty. 
$$
Therefore, 
\begin{equation} \label{Eq: 3}
	\norm{(Kx)'}_\infty \le C_1 \norm{x}_\infty.
\end{equation}

\subsection{Interpolatory Projection}
Denote the uniform partition of $[0,1]$ as
$$
\Delta^{(n)} := \{ 0 =t_0 < t_1 < \cdots <t_n=1\},
$$ 
where $t_j = \displaystyle{\frac{j}{n}}, \quad j=1,2, \ldots, n$. Denote $\Delta_{j} = [t_{j-1}, t_j]$. Let $r$ be a non-negative integer, and define
$$
X_n = \left\{x \in L^\infty[0, 1] : x|_{\Delta_{j}} \text{ is a polynomial of degree } \leq 2r \right\}. 
$$
In each sub-interval $[t_{j-1}, t_j]$, consider $2r+1$ interpolation points as
$$
\tau_j^k = t_{j-1} + \frac{kh}{2r}, \quad k = 0, 1, 2, \ldots, 2r.
$$
That is, the distance between any two consecutive interpolation points is $\frac{h}{2r}$. Observe that the dimension of approximating space $X_n$ is $n(2r+1)$.\\
The interpolatory operator $P_n : X \to X_n$ is defined as:
\begin{equation} \label {Eq: 4}
	P_nx(\tau_j^k) = x(\tau_j^k), \quad k = 0, 1, 2, \ldots, 2r; \quad j = 1, 2, \ldots, n,
\end{equation}
Note that \[P_nx \to x \quad \text{for all} \; x \in C[0,1].\]
Using the Hahn-Banach extension theorem, $P_n$ can be extended to $L^\infty[0,1]$, and then $P_n : L^\infty[0,1] \to X_n$ is a projection.

Let $x \in C[0, 1]$ and $\left\{\tau_j^0, \tau_j^1, \ldots, \tau_j^{2r} \right\}$ be distinct points in $\Delta_{j}$. Denote the first order divided difference of $x$ at $\tau_j^0$ as $\left[\tau_j^0\right]x = x(\tau_j^0)$ and 
the $(2r+1)^{\text{th}}$ divided difference at $\tau_j^0, \tau_j^1, \ldots, \tau_j^{2r}$ as $\left[\tau_j^0, \tau_j^1, \ldots, \tau_j^{2r}\right]{x}$. Then
$$
\left[\tau_j^0, \tau_j^1, \ldots, \tau_j^{2r}\right] x= \frac{\left[\tau_j^1, \tau_j^2, \ldots, \tau_j^{2r}\right] x - \left[\tau_j^0, \tau_j^1, \ldots, \tau_j^{2r-1}\right] x}{\tau_j^{2r}- \tau_j^0},
$$ 
and so on. Let 
$$
\Psi_j(t) = \left(t - \tau_j^0 \right)\left(t - \tau_j^1 \right) \cdots \left(t - \tau_j^{2r} \right), \quad j =1,2, \ldots, n.
$$
Since the polynomial $P_{n,j}x$ interpolates $x$ at the points $\tau_j^0, \tau_j^1, \ldots, \tau_j^{2r}$,
$$
x(t) - P_{n,j}x(t) = \Psi_j(t) \left[\tau_j^0, \tau_j^1, \ldots, \tau_j^{2r}, t  \right]x , \quad t \in \Delta_j.
$$
If $x \in C^{2r+1}(\Delta_{j})$, then we have
\begin{equation} \label{Eq: 5}
	\norm{(I - P_{n,j})x}_\infty \le C_2 \norm{x^{(2r+1)}}_\infty h^{2r+1},
\end{equation}
where $x^{(2r+1)}$ denotes the $(2r+1)^{\text{th}}$ derivative of $x$ and $C_2$ is a constant independent of $h$.

\subsection{Even degree polynomial interpolation}

The case of interpolatory projection at equidistant collocation points (not the Gauss points) for Green's type kernel onto a space of piecewise polynomials of degree $\leq 2r$ with respect to a uniform partition of $[0, 1]$ has not been considered in the research literature. We study the case here.

We first prove two important lemmas based on the divided difference of $Kx$ at $\tau_j^0, \tau_j^1, \ldots, \tau_j^{2r}$ and $s$, denoted by $\left[\tau_j^0, \tau_j^1, \ldots, \tau_j^{2r}, s  \right]Kx$.

\begin{lemma} \label{le1}
	Let $r \ge 0$ and $\left\{\tau_j^0, \tau_j^1, \ldots, \tau_j^{2r} \right\}$ be the set of collocation points in $[t_{j-1}, t_j]$ for $j = 1, 2, \ldots, n$. Then
	\begin{eqnarray*}
		\sup_{s \in [t_{j-1}, t_j]}	\left|\left[\tau_j^0, \tau_j^1, \ldots, \tau_j^{2r}, s  \right]K x\right| \leq C_4 \norm{x}_\infty h^{-2r},
	\end{eqnarray*}
	for some constant $C_4$.
\end{lemma}

\begin{proof} Recall that
	$$
	\tau_j^i = t_{j-1} + \frac{ih}{2r} = t_{j-1}+h\zeta_i, ~ \text{ for } i = 0, 1, 2, \ldots, 2r; ~ j = 1, 2, \ldots, n,
	$$
	where $\zeta_{i} = \frac{i}{2r} \in [0, 1]$.
	Clearly $\tau_j^i \in [t_{j-1},t_j]$, we have
	\begin{equation*}
		(Kx)(s)=\int_0^1  \kappa(s,t) x(t) ~dt,
		~ \text{ with }~
		\kappa(s,t) = 
		\begin{cases}
			\kappa_1(s,t)    ~~~          & 0\leq t\leq s\leq 1,\\
			\kappa_2(s,t)    ~~~          & 0\leq s\leq t\leq 1.\\
		\end{cases} 
	\end{equation*} 
	Let $s \in [t_{j-1},t_j]$, then
	\begin{align} \label{Eq: 6}
		[\tau_j^0,\tau_j^1,s]Kx&=\int_0^1 [\tau_j^0,\tau_j^1,s] \kappa(.,t) x(t) ~dt \nonumber \\
		&= \underset{k \ne j}{\sum_{k=1}^{n}}~\int_{t_{k-1}}^{t_k} [\tau_j^0,\tau_j^1,s] \kappa(.,t) x(t) ~dt + \int_{t_{j-1}}^{t_j} [\tau_j^0,\tau_j^1,s] \kappa(.,t) x(t) ~dt. 
	\end{align}
	Since $\kappa$ is sufficiently differentiable in the interval $[t_{k-1},t_k]$ for $k \ne j$,
	\begin{equation}\label{Eq: 7}
		\left | ~\int_{t_{k-1}}^{t_k} [\tau_j^0,\tau_j^1,s] \kappa(.,t) x(t) ~dt \right | \le M_2 \norm{x}_\infty h.
	\end{equation}  
	For $\tau_j^0, \tau_j^1, s \in [t_{j-1},t_j]$, we need to find bound for $[\tau_j^0,\tau_j^1,s] \kappa(.,t)$. Fix $s \in [t_{j-1},{t_j}].$ \\
	\noindent
	\textbf{Case (I)}
	Whenever $t_{j-1} \le s < \tau_j^0$. Note that
	\begin{equation*}
		[\tau_j^0,s]\kappa(.,t) =\dfrac{\kappa(s,t)-k(\tau_j^0,t)}{s-\tau_j^0} ~=~
		\begin{cases}
			\dfrac{\kappa_1(s,t)-\kappa_1(\tau_j^0,t)}{s-\tau_j^0}  \quad &\text{if}\; t_{j-1} \le t \le s,\\
			\dfrac{\kappa_2(s,t)-\kappa_1(\tau_j^0,t)}{s-\tau_j^0}  \quad &\text{if}\; s < t < \tau_j^0,\\
			\dfrac{\kappa_2(s,t)-\kappa_2(\tau_j^0,t)}{s-\tau_j^0}  \quad &\text{if}\; \tau_j^0 \le t \le t_j.
		\end{cases}
	\end{equation*}	
	Thus, for a fixed $s \in [t_{j-1},\tau_j^0]$, the function $[\tau_j^0,s]\kappa(.,t)$ is continuous on $[t_{j-1},t_j].$
	Since $\kappa_1(.,t)$ and $\kappa_2(.,t)$ are continuous on $[s, \tau_j^0]$ and differentiable on $(s, \tau_j^0)$, by mean value theorem, we have
	$$
	\frac{\kappa_1(s,t)-\kappa_1(\tau_j^0,t)}{s-\tau_j^0} = D^{(1,0)}\kappa_1(\eta_1,t), \;\; \text{and} \;\; \frac{\kappa_2(s,t)-\kappa_2(\tau_j^0,t)}{s-\tau_j^0} = D^{(1,0)}\kappa_2(\eta_2,t),
	$$
	for some $\eta_1, \eta_2 \in (s, \tau_j^0).$
	Now, for $t_{j-1} \le t \le s$,
	$$
	\bigg | \frac{\kappa_1(s,t)-\kappa_1(\tau_j^0,t)}{s-\tau_j^0} \bigg| = \left| \ D^{(1,0)}\kappa_1(\eta_1,t)\right | \le \bigg (\sup_{t_{j-1} \le t \le s \le t_j} \left| D^{(1,0)} \kappa_1(s,t)\right| \bigg ) = M_1. 
	$$
	Similarly, for $\tau_j^0 \le t \le t_j$,
	$$
	\bigg | \frac{\kappa_2(s,t)-\kappa_2(\tau_j^0,t)}{s-\tau_j^0} \bigg| = \left| \ D^{(1,0)}\kappa_2(\eta_2,t)\right | \le \bigg (\sup_{t_{j-1} \le s \le t \le t_j} \left| D^{(1,0)} \kappa_2(s,t)\right| \bigg ) = M_1. 
	$$
	Hence, 
	\begin{equation}\label{Eq: 8}
		\left|[\tau_j^0,s]\kappa(.,t)\right| \le M_1, ~ \; \;\; \text{if} \; t_{j-1} \le t \le s \; \text{and} \; \tau_j^0 \le t \le t_j.
	\end{equation}
	Now, for $s < t < \tau_j^0$, first we find
	\begin{align*}
		\frac{\kappa_2(s,t)-\kappa_1(\tau_j^0,t)}{s-\tau_j^0} &= \frac{\kappa_2(s,t)-\kappa_2(t,t)+\kappa_1(t,t)-\kappa_1(\tau_j^0,t)}{s-\tau_j^0} \\
		&= \frac{D^{(1,0)}\kappa_2(\eta_3,t)(s-t)}{s-\tau_j^0}+\frac{D^{(1,0)}\kappa_1(\eta_4,t)(t-\tau_j^0)}{s-\tau_j^0},
	\end{align*}
	for some $\eta_3 \in (s,t)$ and $\eta_4 \in (t, \tau_j^0).$\\
	Hence, 
	\begin{align*}
		\bigg | \frac{\kappa_2(s,t)-\kappa_1(\tau_j^0,t)}{s-\tau_j^0} \bigg| &= \bigg| \frac{D^{(1,0)}\kappa_2(\eta_3,t)(s-t)}{s-\tau_j^0}+\frac{D^{(1,0)}\kappa_1(\eta_4,t)(t-\tau_j^0)}{s-\tau_j^0}\bigg| \nonumber \leq 2M_1.
	\end{align*}
	Therefore $\displaystyle{\left|[\tau_j^0,s]\kappa(.,t)\right| \le 2M_1}$ if $ s < t < \tau_j^0 $.
	Thus, by \eqref{Eq: 8} and the above inequality, we obtain
	\begin{equation}\label{Eq: 9}
		\sup_{t \in [t_{j-1},t_j]}\left|[\tau_j^0,s]\kappa(.,t)\right| \le 2M_1.
	\end{equation}
	
	\noindent	
	\textbf{Case (II)}
	Whenever $s = \tau_j^0$. In this case,
	\begin{align*}
		[\tau_j^0,\tau_j^0] \kappa(.,t) = \begin{cases}
			D^{(1,0)}\kappa_1(\tau_j^0,t)  \quad &\text{if}\; t_{j-1} \le t < \tau_j^0 < 1,\\
			D^{(1,0)}\kappa_2(\tau_j^0,t)  \quad &\text{if}\; t_{j-1} \le \tau_j^0 < t  < 1.\\	
		\end{cases} 
	\end{align*}
	Hence, 
	\begin{equation}\label{Eq: 10}
		\sup_{t \in [t_{j-1},t_j]}\left|[\tau_j^0,\tau_j^0]\kappa(.,t)\right| \le M_1.
	\end{equation}
	
	\noindent
	\textbf{Case (III)}
	Whenever $\tau_j^0<s \le t_j$. Proceeding the same way as in Case (I), we have
	\begin{equation*}
		\sup_{t \in [t_{j-1},t_j]}\left|[\tau_j^0,s]\kappa(.,t)\right| \le 2M_1.
	\end{equation*}
	Hence, by \eqref{Eq: 9}, \eqref{Eq: 10} and the above inequality, we obtain
	\begin{equation}\label{Eq: 11}
		\sup_{s,t \in [t_{j-1},t_j]}\left|[\tau_j^0,s]\kappa(.,t)\right| \le 2M_1.
	\end{equation}
	Note that
	$$
	[\tau_j^0,\tau_j^1,s]\kappa(.,t) =\dfrac{[\tau_j^0,s]\kappa(.,t)-[\tau_j^1,s]\kappa(.,t)}{\tau_j^0-\tau_j^1} 
	= \dfrac{1}{h(\zeta_0-\zeta_1)} \bigg [ [\tau_j^0,s]\kappa(.,t)-[\tau_j^1,s]\kappa(.,t) \bigg ],
	$$
	which implies
	$$
	\bigg |[\tau_j^0,\tau_j^1,s]\kappa(.,t) \bigg| \le \dfrac{1}{h |\zeta_0-\zeta_1|} \bigg [ \left |[\tau_j^0,s]\kappa(.,t)\right |+\left |[\tau_j^1,s]\kappa(.,t) \right | \bigg ].
	$$
	From \eqref{Eq: 11}, it follows that
	$$
	\sup_{s,t \in [t_{j-1},t_j]}\left|[\tau_j^0,\tau_j^1,s]\kappa(.,t)\right| \le \dfrac{4M_1}{h |\zeta_0-\zeta_1|}.
	$$
	Therefore, 
	$$
	\bigg |\int_{t_{j-1}}^{t_j} [\tau_j^0,\tau_j^1,s] \kappa(.,t) x(t) ~dt \bigg | 
	\le \dfrac{4M_1}{h |\zeta_0-\zeta_1|}\norm{x}_\infty =  \dfrac{C_3}{h}\norm{x}_\infty,
	$$
	where $C_3=\dfrac{4M_1}{|\zeta_0-\zeta_1|}$ is some constant.
	Hence, using \eqref{Eq: 7} and above inequality in \eqref{Eq: 6}, we obtain
	$$
	\left|[\tau_j^0,\tau_j^1,s]Kx \right| \le \bigg ( M_2 h + \dfrac{C_3}{h}\bigg)\norm{x}_\infty \le \dfrac{C_4}{h}\norm{x}_\infty,
	$$
	where $C_4$ is some constant. This gives the required estimate
	$$
	\left|[\tau_j^0,\tau_j^1,\ldots, \tau_j^{2r}, s]Kx \right| \le C_5 \norm{x}_\infty \dfrac{1}{h^{2r}},
	$$
	for some constant $C_5$. Since $s \in [t_{j-1},t_j]$ is arbitrary, the proof is complete.
\end{proof}

\begin{lemma} \label{le2}
	Let $r \ge 0$ and $\left\{\tau_j^0, \tau_j^1, \ldots, \tau_j^{2r} \right\}$ be the set of collocation points in $[t_{j-1}, t_j]$ for $j = 1, 2, \ldots, n$. Then
	\begin{eqnarray*}
		\sup_{s \in [t_{j-1}, t_j]}	\left|\left[\tau_j^0, \tau_j^1, \ldots, \tau_j^{2r}, s, s \right]K x\right| \leq C_8 \norm{x}_\infty h^{-2r},
	\end{eqnarray*}
	for some constant $C_8$.
\end{lemma}

\begin{proof}
	Let $s \in [t_{j-1},t_j]$, then
	\begin{align}\label{Eq: 12}
		[\tau_j^0,s,s]Kx&=\int_0^1 [\tau_j^0,s,s] \kappa(.,t) x(t) ~dt \nonumber \\
		&= \underset{k \ne j}{\sum_{k=1}^{n}}~\int_{t_{k-1}}^{t_k} [\tau_j^0,s,s] \kappa(.,t) x(t) ~dt + \int_{t_{j-1}}^{t_j} [\tau_j^0,s,s] \kappa(.,t) x(t) ~dt.
	\end{align}
	Since $\kappa$ is sufficiently differentiable in the interval $[t_{k-1},t_k]$ for $k \ne j$,
	\begin{equation}\label{Eq: 13}
		\left |~ \int_{t_{k-1}}^{t_k} [\tau_j^0,s,s] \kappa(.,t) x(t) dt \right | \le M_2 \norm{x}_\infty h.
	\end{equation} 
	For $\tau_j^0, s \in [t_{j-1},t_j]$, we need to find bound for $[\tau_j^0,s,s] \kappa(.,t)$.
	Fix $s \in [t_{j-1},{t_j}].$
	
	\noindent
	\textbf{Case (I)}
	Whenever $t_{j-1} \le s < \tau_j^0$. Note that
	$$
	[\tau_j^0,s,s]\kappa(.,t) =\dfrac{[s,s]\kappa(.,t)-[\tau_j^0,s]\kappa(.,t)}{s-\tau_j^0}.
	$$
	Thus,
	\begin{equation*}
		[\tau_j^0,s,s]\kappa(.,t) ~=~
		\begin{cases}
			\dfrac{\frac{\partial \kappa_1}{\partial s}(s,t)- \left ( \frac{\kappa_1(s,t)-\kappa_1(\tau_j^0,t)}{s-\tau_j^0}\right) }{s-\tau_j^0}  \quad &\text{if}\; t_{j-1} \le t \le s,\\
			\dfrac{\frac{\partial \kappa_2}{\partial s}(s,t)- \left ( \frac{\kappa_2(s,t)-\kappa_1(\tau_j^0,t)}{s-\tau_j^0}\right) }{s-\tau_j^0}  \quad &\text{if}\; s < t < \tau_j^0, \\
			\dfrac{\frac{\partial \kappa_2}{\partial s}(s,t)- \left ( \frac{\kappa_2(s,t)-\kappa_2(\tau_j^0,t)}{s-\tau_j^0}\right) }{s-\tau_j^0}  \quad &\text{if}\; \tau_j^0 \le t \le t_j.
		\end{cases}
	\end{equation*}
	This function $[\tau_j^0,s,s]\kappa(.,t)$ is possibly discontinuous at $t= s$.
	Since $\kappa_1(.,t)$ and $\kappa_2(.,t)$ are continuous on $[s, \tau_j^0]$ and differentiable on $(s, \tau_j^0)$, by mean value theorem, we obtain
	$$
	\frac{\kappa_1(s,t)-\kappa_1(\tau_j^0,t)}{s-\tau_j^0} = D^{(1,0)}\kappa_1(\nu_1,t), \;\; \text{and} \;\; \frac{\kappa_2(s,t)-\kappa_2(\tau_j^0,t)}{s-\tau_j^0} = D^{(1,0)}\kappa_2(\nu_2,t),
	$$
	for some $\nu_1, \nu_2 \in (s, \tau_j^0)$.	Now, for $t_{j-1} \le t \le s$,
	\begin{align*}
		\Bigg | \dfrac{\frac{\partial \kappa_1}{\partial s} (s,t)- \left ( \frac{\kappa_1(s,t)-\kappa_1(\tau_j^0,t)}{s-\tau_j^0}\right )}{s-\tau_j^0} \Bigg | &= \Bigg | \dfrac{ D^{(1,0)}\kappa_1(s,t)- D^{(1,0)}\kappa_1(\nu_1,t)}{s-\tau_j^0} \Bigg | \nonumber \\
		&= \Bigg | \dfrac{ D^{(2,0)}\kappa_1(\nu_3,t)(s-\nu_1)}{s-\tau_j^0} \Bigg |,~  \; \text{for some} \; \nu_3 \in (s, \nu_1) \nonumber \\
		&\le \bigg (\sup_{t_{j-1} \le t \le s \le t_j} \left| D^{(2,0)} \kappa_1(s,t)\right| \bigg ) = M_2. 
	\end{align*}
	Similarly, it can be shown that if $\tau_j^0 \le t \le t_j$, then
	$$
	\Bigg | \dfrac{\frac{\partial \kappa_2}{\partial s}(s,t) - \left ( \frac{\kappa_2(s,t)-\kappa_2(\tau_j^0,t)}{s-\tau_j^0}\right )}{s-\tau_j^0} \Bigg | \: \le \: \bigg (\sup_{t_{j-1} \le s \le t \le t_j} \left| D^{(2,0)} \kappa_2(s,t)\right| \bigg )= M_2.
	$$
	Hence, \begin{equation}\label{Eq: 14}
		\left|[\tau_j^0,s,s]\kappa(.,t)\right| \le M_2, ~ \; \;\; \text{if} \; t_{j-1} \le t \le s \; \text{and} \; \tau_j^0 \le t \le t_j.
	\end{equation}
	Now, for $s < t < \tau_j^0$, first we find
	\begin{align*}
		\frac{\kappa_2(s,t)-\kappa_1(\tau_j^0,t)}{s-\tau_j^0} &= \frac{\kappa_2(s,t)-\kappa_2(t,t)+\kappa_1(t,t)-\kappa_1(\tau_j^0,t)}{s-\tau_j^0} \\
		&= \frac{D^{(1,0)}\kappa_2(\nu_5,t)(s-t)}{s-\tau_j^0}+\frac{D^{(1,0)}\kappa_1(\nu_6,t)(t-\tau_j^0)}{s-\tau_j^0},
	\end{align*}
	for some $\nu_5 \in (s,t)$ and $\nu_6 \in (t, \tau_j^0)$.
	Then, 
	\begin{multline*}
		\Bigg | \frac{\partial \kappa_2}{\partial s}(s,t)- \left ( \frac{\kappa_2(s,t)-\kappa_1(\tau_j^0,t)}{s-\tau_j^0}\right ) \Bigg | \\ = \Bigg | D^{(1,0)}\kappa_2(s,t)- \frac{D^{(1,0)}\kappa_2(\nu_5,t)(s-t)}{s-\tau_j^0}-\frac{D^{(1,0)}\kappa_1(\nu_6,t)(t-\tau_j^0)}{s-\tau_j^0}\Bigg |,
	\end{multline*}
	which gives
	$$
	\Bigg | \frac{\partial \kappa_2}{\partial s}(s,t)- \left ( \frac{\kappa_2(s,t)-\kappa_1(\tau_j^0,t)}{s-\tau_j^0}\right ) \Bigg | \le M_1 +M_1+M_1=3M_1.
	$$
	Hence,
	$$
	\bigintsss_{s}^{\tau_j^0}\Bigg | \dfrac{\frac{\partial \kappa_2}{\partial s}(s,t)- \left ( \frac{\kappa_2(s,t)-\kappa_1(\tau_j^0,t)}{s-\tau_j^0}\right) }{s-\tau_j^0} \Bigg |~dt \le 3M_1.
	$$
	Using \eqref{Eq: 14} and the above inequality, we obtain
	$$
	\bigg |\int_{t_{j-1}}^{t_j} [\tau_j^0,s,s] \kappa(.,t) x(t) ~dt \bigg | 
	\le \left ( M_2 +3 M_1 + M_2 \right )\norm{x}_\infty = \left ( 3 M_1  + 2 M_2 \right )\norm{x}_\infty.
	$$ 
	Therefore, using \eqref{Eq: 13} and the above inequality in \eqref{Eq: 12} we have
	\begin{equation}\label{Eq: 15}
		\left |[\tau_j^0,s,s]Kx \right| \le \left ( M_2 h +3 M_1  + 2 M_2 \right )\norm{x}_\infty \le C_5 \norm{x}_\infty,
	\end{equation}
	where $C_5$ is some constant.
	
	\noindent
	\textbf{Case (II)}
	Whenever $s = \tau_j^0$.
	\begin{align*}
		[\tau_j^0,\tau_j^0,\tau_j^0]Kx&= \frac{1}{2} (Kx)^{''}(\tau_j^0)\\
		&= \frac{1}{2}\left ( \frac{\partial \kappa_1}{\partial s}(\tau_j^0,\tau_j^0)-\frac{\partial \kappa_2}{\partial s}(\tau_j^0,\tau_j^0)\right )x(\tau_j^0)\\&+\frac{1}{2}\left [~ \int_{t_j-1}^{\tau_j^0}\frac{\partial^2 \kappa_1}{\partial^2 s}(\tau_j^0,t)x(t)~dt+\int_{\tau_j^0}^{t_j}\frac{\partial^2 \kappa_2}{\partial^2 s}(\tau_j^0,t)x(t)~dt\right ]. 
	\end{align*}
	Hence, 
	\begin{equation}\label{Eq: 16}
		\left |[\tau_j^0,\tau_j^0,\tau_j^0]Kx\right| \le \left ( M_1+ M_2\right )\norm{x}_\infty. 
	\end{equation}

	\noindent
	\textbf{Case (III)}
	Whenever $\tau_j^0<s \le t_j$. Similar in Case (I), we have
	$$
	\left |[\tau_j^0,s,s]Kx \right| \le \left ( M_2 h +3 M_1  + 2 M_2 \right )\norm{x}_\infty \le C_6 \norm{x}_\infty, 
	$$
	where $C_6$ is some constant. From \eqref{Eq: 15}, \eqref{Eq: 16} and above inequality, it follows that
	\begin{equation}\label{Eq: 17}
		\sup_{s \in [t_{j-1},t_j]} \left|[\tau_j^0,s,s]Kx \right| \le C_6 \norm{x}_\infty,
	\end{equation}
	which proves the required estimate for the case $r = 0$. Now, for $r=1$,
	$$
	[\tau_j^0,\tau_j^1,s,s]Kx  = \dfrac{[\tau_j^0,s,s]Kx-[\tau_j^1,s,s]Kx}{\tau_j^0-\tau_j^1} = \frac{1}{h (\zeta_0-\zeta_1) } \bigg [[\tau_j^0,s,s]Kx-[\tau_j^1,s,s]Kx \bigg],
	$$
	which implies
	$$
	\left |[\tau_j^0,\tau_j^1,s,s]Kx \right | 	\le \frac{1}{h | \zeta_0-\zeta_1 |} \bigg [\left|[\tau_j^0,s,s]Kx\right|+\left|[\tau_j^1,s,s]Kx\right| \bigg].
	$$
	Hence, from \eqref{Eq: 17}, we obtain
	$$
	\sup_{s \in [t_{j-1},t_j]} \left|[\tau_j^0,\tau_j^1,s,s]Kx \right| \le \frac{2C_6}{h | \zeta_0-\zeta_1 |} \norm{x}_\infty = C_7 \norm{x}_\infty \frac{1}{h},
	$$
	where $C_7= \frac{2C_6}{| \zeta_0-\zeta_1 |}$ is some constant.
	By mathematical induction, it follows that
	$$
	\sup_{s \in [t_{j-1},t_j]} \left|[\tau_j^0,\tau_j^1,\ldots, \tau_j^{2r},s,s]Kx\right| \le C_8 \norm{x}_\infty \frac{1}{h^{2r}} \;,
	$$
	where $C_8$ is some constant.
\end{proof}

\section{Iteration methods of approximation}
Recall that $K$ is a Fredholm integral operator with Green's function type kernel and $P_n$ is the interpolatory operator at $2r+1$ interpolating points defined by \eqref{Eq: 4}. In the collocation method, the main equation \eqref{Eq: 2} is approximated by
\begin{equation} \label{Eq: 18}
	\phi_{n}^C - P_nK\phi_{n}^C = P_nf.
\end{equation}
Sloan \cite{Sl1} introduced iterated collocation solution by taking one step iteration as
\begin{equation} \label{Eq: 19}
	\phi_{n}^S = K\phi_{n}^C + f.
\end{equation}
The following modified collocation method has been introduced by Kulkarni \cite{RPK1}:
\begin{equation} \label{Eq: 20}
	\phi_{n}^M - K_{n}^M\phi_{n}^M = f,
\end{equation}
where
\[	K_n^M = P_nK + KP_n - P_nKP_n.\]
The one step iteration is defined as
\begin{equation} \label{Eq: 21}
	\tilde{\phi}_n^M=K\phi_n^M+f,
\end{equation}
which gives the solution for iterated modified collocation method.

Now, we establish preliminary results below using the standard technique along with the aforementioned lemmas, and these outcomes will be useful for the main result.
\begin{proposition} \label{p1}
	If $x \in C^{2r+2}([0,1])$, then
	\begin{eqnarray*}
		\norm{K(I-P_n)x }_\infty \le 2C_1 \norm{x}_{2r+2, \infty} h^{2r + 2},
	\end{eqnarray*}
	where $C_1$ is a constant independent of $h$.
\end{proposition}
\begin{proof} Let $s \in [0,1]$. Then we have
	\begin{align*}
		K(I-P_n)x(s) &=\int_0^1  \kappa(s,t) (I-P_n) x(t) ~dt \notag \\
		&= \sum_{j=1}^{n} \int_{t_{j-1}}^{t_j}  \kappa(s,t) (I-P_{n,j})x(t) ~dt  \notag \\
		&= \sum_{j=1}^{n ~}\int_{t_{j-1}}^{t_j}  \kappa(s,t) \Psi_j(t)  \left[\tau_j^0,\tau_j^1,\ldots, \tau_j^{2r},t\right]x  ~dt,
	\end{align*}
	where
	$$
	\Psi_j(t) = \prod_{i=0}^{2r} (t- \tau_j^i), \quad \text{for} ~ j=1,2,\ldots,n.
	$$ 
	Let $s \in [t_{i-1},t_i] \subset [0,1]$, for some $1 \le i \le n$. Then we write
	\begin{align} \label{Eq: 22}
		K(I-P_n)x(s) &= \sum_{\substack{j=1 \\  j\ne i} }^{n} ~	\int_{t_{j-1}}^{t_j}  \kappa(s,t) \Psi_j(t) \left[\tau_j^0,\tau_j^1,\ldots, \tau_j^{2r},t\right]x ~dt \notag \\
		& \quad + \int_{t_{i-1}}^{t_i}  \kappa(s,t) \Psi_i(t) \left[\tau_i^0,\tau_i^1,\ldots, \tau_i^{2r},t\right]x ~dt 		
	\end{align}
	For fix $s$, let
	\begin{align*}
		f_j^s(t) =  \kappa(s,t) \left[\tau_j^0,\tau_j^1,\ldots, \tau_j^{2r},t\right]x, \quad t \in [t_{j-1}, t_j], ~ \text{for} ~ j=1,2,\ldots,n.
	\end{align*}
	That is
	\begin{equation*}
		f_j^s(t) ~= ~
		\begin{cases}
			\kappa_1(s,t) \left[\tau_j^0,\tau_j^1,\ldots, \tau_j^{2r},t\right]x  ~~~         &  j \le i,\\
			\kappa_2(s,t) \left[\tau_j^0,\tau_j^1,\ldots, \tau_j^{2r},t\right]x ~~~          &  j \ge i.\\
		\end{cases} 
	\end{equation*} 
	Note that the function $f_j^s$ is continuous on $[t_{j-1}, t_j]$ and differentiable on $(t_{j-1}, t_j)$ for $j \ne i$. Therefore, by Taylor series expansion of $f_j^s$ about the point $\frac{t_{j-1}+t_j}{2}$, there exist a point $\zeta_{j1} \in (t_{j-1}, t_j)$ such that
	$$
	(f_j^s)(t) = f_j^s \left(\frac{t_{j-1}+t_j}{2}\right) + \left(t - \frac{t_{j-1}+t_j}{2}\right) (f_j^s)'(\zeta_{j1}),
	$$ 
	where
	\begin{equation*}
	(f_j^s)'(\zeta_{j1}) ~= ~
	\begin{cases}
	\kappa_1(s,\zeta_{j1}) \left[\tau_j^0,\tau_j^1,\ldots, \tau_j^{2r},\zeta_{j1}, \zeta_{j1} \right]x \\
	+ D^{(0,1)} \kappa_1(s,\zeta_{j1}) \left[\tau_j^0,\tau_j^1,\ldots, \tau_j^{2r},\zeta_{j1} \right]x~~~          &  j \le i,\vspace*{1.5mm}\\ 
		\kappa_2(s,\zeta_{j1}) \left[\tau_j^0,\tau_j^1,\ldots, \tau_j^{2r},\zeta_{j1}, \zeta_{j1} \right]x \\
	+  D^{(0,1)} \kappa_2(s,\zeta_{j1}) \left[\tau_j^0,\tau_j^1,\ldots, \tau_j^{2r},\zeta_{j1} \right]x~~~          & j \ge i.\\
	\end{cases}
	\end{equation*}
	Since $x \in C^{2r+2}([0,1])$ , then recall that
	$$
	\norm{x}_{2r + 2, \infty} = \max_{0 \le i \le 2r +2} \norm{x^{(i)}}_\infty = \max_{0 \le i \le 2r +2} \left\{  \sup_{t \in [0,1]} \left|x^{(i)}(t)\right| \right\}.
	$$   
	Thus 
	$$
	\left|(f_j^s)'(\zeta_{j1}) \right| \le 2 C_1 \norm{x}_{2r +2, \infty}.
	$$
	Now, we have
	$$
	\int_{t_{j-1}}^{t_j}  \kappa(s,t) \Psi_j(t) \left[\tau_j^0,\tau_j^1,\ldots, \tau_j^{2r},t\right]x ~dt = \int_{t_{j-1}}^{t_j}  \left(t - \frac{t_{j-1}+t_j}{2}\right) (f_j^s)'(\zeta_{j1}) \Psi_j(t) ~dt.
	$$
	It follows that 
	$$
	\left|\int_{t_{j-1}}^{t_j}  \kappa(s,t) \Psi_j(t) \left[\tau_j^0,\tau_j^1,\ldots, \tau_j^{2r},t\right]x ~dt \right| \le  2 C_1 \norm{x}_{2r +2, \infty} h^{2r + 3}.
	$$
	Since $\displaystyle{\int_{t_{i-1}}^{t_i} \Psi_i(t) ~dt =0}$,
	\begin{align*}
		\int_{t_{i-1}}^{t_i}  \kappa(s,t) \Psi_i(t) \left[\tau_i^0,\tau_i^1,\ldots, \tau_i^{2r},t\right]x ~dt &= \int_{t_{i-1}}^{t_i}   (f_i^s)(t) \Psi_i(t) ~dt, \quad s \in [t_{i-1},t_i] \\
		&= \int_{t_{i-1}}^{t_i}   [(f_i^s)(t)-(f_i^s)(s)] \Psi_i(t) ~dt.
	\end{align*}
	Since the function $f_i^s$ is continuous on $[t_{i-1}, t_i]$ and differentiable on $(t_{i-1}, t_i)$, by mean value theorem we have 
	\begin{equation*}
		(f_i^s)(t)-(f_i^s)(s) = (t-s)(f_j^s)'(\zeta_{j2}),
	\end{equation*}
	for some $\zeta_{j2} \in (t_{i-1}, t_i)$. Then
	\begin{equation*}
		\left|~\int_{t_{i-1}}^{t_i}  \kappa(s,t) \Psi_i(t) \left[\tau_i^0,\tau_i^1,\ldots, \tau_i^{2r},t\right]x ~dt \right| \le  2 C_1 \norm{x}_{2r +2, \infty} h^{2r + 3}.
	\end{equation*}
	By \eqref{Eq: 22},
	\begin{align*}
		\left|K(I-P_n)x(s)\right| &\le \sum_{\substack{j=1 \\  j\ne i} }^{n} ~	\left|~\int_{t_{j-1}}^{t_j}  \kappa(s,t) \left( [\tau_j^0,\tau_j^1,\ldots, \tau_j^{2r},t]x \right) \Psi_j(t) ~dt \right|\\
		& \quad + \left|~\int_{t_{i-1}}^{t_i}  \kappa(s,t) \left( [\tau_i^0,\tau_i^1,\ldots, \tau_i^{2r},t]x \right) \Psi_i(t) ~dt 	\right| .
	\end{align*}
	Hence
	\begin{equation*}
		\norm{K(I-P_n)x}_\infty \le 2 C_1 \norm{x}_{2r +2, \infty} h^{2r + 2},
	\end{equation*}
	which proves the result.
\end{proof}
\begin{proposition} \label{p2}
	If $x \in C^{2r+2}([0,1])$, then
	\begin{equation*} 	
		\norm{(I-P_n)K(I-P_n)x}_\infty =  
		\begin{cases}
			O(h^{2r+2})~~~          &  r \ge 1,\\
			O(h^3)~~~          & r=0.\\
		\end{cases} 
	\end{equation*}
\end{proposition}	

\begin{proof}
	For $r \ge 1$, whenever $x \in C^{2r+2}([0,1])$, note that
	$$
	\norm{(I-P_n)K(I-P_n)x}_\infty \le \norm{I-P_n} \norm{K(I-P_n)x}_\infty.
	$$
	Since $\norm{P_n}$ is uniformly bounded that means $\norm{P_n} \le p$, for $p > 0$ and by Proposition \eqref{p1}, we have
	$$
	\norm{(I-P_n)K(I-P_n)x}_\infty =O(h^{2r+2}).
	$$
	If $r=0$, using \eqref{Eq: 5} we write 
	$$
	\norm{(I-P_n)K(I-P_n)x}_\infty \le C_2 \norm{(K(I-P_n)x)'}_\infty h.
	$$
	From Chatelin-Lebbar \cite[Theorem 15]{Cha-Leb2}, we have 
	$$
	\norm{(K(I-P_n)x)'}_\infty = O(h^2).
	$$
	Combining the above two estimates, we obtain
	$$
	\norm{(I-P_n)K(I-P_n)x}_\infty = O(h^3).
	$$
\end{proof}
\begin{proposition} \label{p3}
	If $x \in C^{2r+1}([0,1]),$ then
	\begin{equation*}
		\norm{K(I-P_n)K(I-P_n)x }_\infty = \begin{cases}
			O(h^{2r+3})~~~          &  r \ge 1,\\
			O(h^4)~~~          & r=0.\\
		\end{cases}
	\end{equation*}
\end{proposition}

\begin{proof}
	For any $s \in [0,1]$, 
	\begin{align} \label{Eq: 23}
		K(I-P_n)K(I-P_n)x(s) &=\int_0^1  \kappa(s,t) (I-P_n)K(I-P_n)x(t) ~dt \notag \\
		&= \sum_{j=1}^{n} \int_{t_{j-1}}^{t_j}  \kappa(s,t) (I-P_{n,j})K(I-P_n)x(t) ~dt  \notag \\
		&= \sum_{j=1}^{n} ~\int_{t_{j-1}}^{t_j}  \kappa(s,t) \left( \left[\tau_j^0,\tau_j^1,\ldots, \tau_j^{2r},t\right]K(I-P_n)x \right) \Psi_j(t) ~dt.
	\end{align}
	Let $s \in [t_{i-1},t_i] \subset [0,1]$, for some $1 \le i \le n$. Let
	\begin{align*}
		g_j^s(t) =  \kappa(s,t) \left( \left[\tau_j^0,\tau_j^1,\ldots, \tau_j^{2r},t\right]K(I-P_n)x \right), \quad t \in [t_{j-1}, t_j], ~ \text{for} ~ j=1,2,\ldots,n.
	\end{align*}
	Note that
	\begin{equation*}
		g_j^s(t) ~= ~
		\begin{cases}
			\kappa_1(s,t) \left( \left[\tau_j^0,\tau_j^1,\ldots, \tau_j^{2r},t\right]K(I-P_n)x \right) ~~~          &  j \le i,\\
			\kappa_2(s,t) \left( \left[\tau_j^0,\tau_j^1,\ldots, \tau_j^{2r},t\right]K(I-P_n)x \right)~~~          & j \ge i.\\
		\end{cases} 
	\end{equation*} 
	Denote
	$$
	\left ( \delta_j^{2r}  K(I-P_n)x \right ) [t] =\left( \left[\tau_j^0,\tau_j^1,\ldots, \tau_j^{2r},t\right]K(I-P_n)x \right),
	$$
	and
	$$
	\left ( \delta_j^{2r}  K(I-P_n)x \right ) [t, t] =\left( \left[\tau_j^0,\tau_j^1,\ldots, \tau_j^{2r},t, t \right]K(I-P_n)x \right).
	$$  
	Note that the function $g_j^s$ is continuous on $[t_{j-1}, t_j]$ and differentiable on $(t_{j-1}, t_j)$ for $j \ne i$. Then, it follows that
	\begin{equation*}
		(g_j^s)'(t) ~=~
		\begin{cases}
			\kappa_1(s,t) \left ( \delta_j^{2r}  K(I-P_n)x \right ) [t, t] + \left ( \delta_j^{2r}  K(I-P_n)x \right ) [t]~ D^{(0,1)} \kappa_1(s,t) ~         &  j \le i,\\
			\kappa_2(s,t) \left ( \delta_j^{2r}  K(I-P_n)x \right ) [t, t] + \left ( \delta_j^{2r}  K(I-P_n)x \right ) [t]~ D^{(0,1)} \kappa_2(s,t)~         & j \ge i.\\
		\end{cases} 
	\end{equation*} 
	Then by Lemma \ref{le1}, Lemma \ref{le2} and the estimate \eqref{Eq: 5}, we have
	$$
	\norm{(g_j^s)'(t)}_\infty \le C_9 C_1 \norm{(I-P_n)x}_\infty h^{-2r}  \le C_9 C_2 C_1 \norm{x^{(2r + 1)}}_\infty h,
	$$
	where $C_9$ is some constant. Replacing $g_j^s$ in (\eqref{Eq: 23}), we write
	\begin{align*}
		K(I-P_n)K(I-P_n)x(s) &= \sum_{\substack{j=1 \\  j\ne i} }^{n}~ 	\int_{t_{j-1}}^{t_j}  g_j^s(t) \Psi_j(t) ~dt + \int_{t_{i-1}}^{t_i}  g_i^s(t) \Psi_i(t) ~dt.
	\end{align*}
	By employing similar argument as presented in the proof of Proposition \ref{p1}, we acquire 
	$$
	\left|K(I-P_n)K(I-P_n)x(s)\right| \le 2 C_9 C_2 C_1 \norm{x^{(2r + 1)}}_\infty h^{2r + 3}.
	$$
	This proves the proposition for case $r\ge 1$.\\
	If $r=0$ and let $x \in C^2([0,1])$, then from Kulkarni-Nidhin \cite[Proposition 3.7]{RPKTJN} it follows that
	$$
	\norm{K(I-P_n)K(I-P_n)x }_\infty = O(h^4).
	$$
	That is our desired estimate.
\end{proof}
One of our main results which determine the order of convergence for both the collocation solution $\phi_n^C$ and the iterated collocation solution $\phi_n^S$ is proved as follows.

\begin{theorem} \label{t4}
	Let $K$ be the Fredholm integral operator defined by \eqref{Eq: 1} with a Green's function type kernel. Let $\phi \in C^{2r+2}([0,1])$ be the unique solution of the equation  \eqref{Eq: 2} and $P_n$ be the interpolatory operator onto the approximating space $X_n$ defined by \eqref{Eq: 4}. Let $\phi_n^C$ be the unique solution of \eqref{Eq: 18} and the iterated collocation solution $\phi_n^S$ is defined by \eqref{Eq: 19}. Then
	\begin{eqnarray*} 
		\norm{\phi - \phi_n^C }_\infty = O\left(h^{2r +1}\right), \quad \norm{\phi - \phi_n^S }_\infty = O\left(h^{2r + 2}\right).
	\end{eqnarray*}
\end{theorem}

\begin{proof}
	We quote the following error estimates from Atkinson \cite[Section 3.4]{KEA0}:
	$$
	\norm{\phi - \phi_n^C }_\infty \leq C_{10} \norm{(I - P_n)\phi}_\infty
	$$
	and
	$$
	\norm{\phi - \phi_n^S }_\infty \leq  C_{10} \norm{K(I - P_n)\phi}_\infty,
	$$
	where $C_{10}$ is a constant independent of $n$. 
	Therefore, by inequality \eqref{Eq: 5} and Proposition \ref{p1}, we have
	$$
	\norm{\phi - \phi_n^C }_\infty = O\left(h^{2r + 1}\right), \quad \norm{\phi - \phi_n^S }_\infty = O\left(h^{2r + 2}\right).
	$$
\end{proof}
We now prove our main result by the following theorem which determine the orders of convergence for both the modified collocation solution $\phi_n^M$ and the iterated modified collocation solution $\tilde{\phi}_n^M$.

\begin{theorem} \label{t5}
	Let $K$ be the Fredholm integral operator defined by \eqref{Eq: 1} with a Green's function type kernel. Let $\phi \in C^{2r+2}([0,1])$ be the unique solution of the integral equation  \eqref{Eq: 2}. Let $P_n$ be the interpolatory operator onto the approximating space $X_n$ defined by \eqref{Eq: 4}. Let $\phi_n^M$ be the unique solution of \eqref{Eq: 20} and the iterated modified collocation solution $\tilde{\phi}_n^M$ is defined by \eqref{Eq: 21}. Then
	\begin{equation*}
		\norm{\phi - \phi_n^M }_\infty =
		\begin{cases}
			O\left(h^{2r + 2}\right)~~~          &  r \ge 1,\\
			O\left(h^{3}\right)~~~          & r=0,\\
		\end{cases} 
	\end{equation*}
	and
	\begin{equation*}
		\norm{\phi - \tilde{\phi}_n^M }_\infty =
		\begin{cases}
			O\left(h^{2r + 3}\right)~~~          &  r \ge 1,\\
			O\left(h^{4}\right)~~~          & r=0.\\
		\end{cases} 
	\end{equation*}
\end{theorem}

\begin{proof}
	Since $K$ is a bounded linear operator and using \eqref{Eq: 5}, we get
	$$
	\norm{K(I - P_n)K(I - P_n) x}_\infty \le C_2 \norm{K} \norm{(K(I - P_n) x)'}_\infty h.
	$$
	It follows from inequality \eqref{Eq: 4} that
	\begin{align*}
		\norm{K(I - P_n)K(I - P_n) x}_\infty
		&\le C_2 C_1  \norm{K} \norm{(I - P_n) x}_\infty h \\
		&\le  C_1 C_2 \norm{K} \norm{I - P_n} \norm{x}_\infty h \\
		&\le C_1 C_2 (1 + p) \norm{K} \norm{x}_\infty h.
	\end{align*}
	Then, 
	\begin{equation}\label{Eq: 24}
		\norm{K(I - P_n)K(I - P_n)} = O(h).
	\end{equation} 
	We quote the following inequality from Kulkarni \cite{RPK1}:\\
	For all large $n$,
	\begin{align*}
		\norm{\phi - \phi_n^M }_\infty \leq C_{11} \norm{(I - P_n)K(I - P_n)\phi}_\infty,
	\end{align*} and
	\begin{align*}
		\norm{\phi - \tilde{\phi}_n^M }_\infty \leq  \norm{(I - K)^{-1}} ( \norm{K(I - P_n)K(I - P_n)\phi}_\infty \\ + \norm{K(I - P_n)K(I - P_n)} \norm{\phi - \phi_n^M }_\infty ),
	\end{align*}
	where $C_{11}$ is a constant independent of $n$.\\
	From Proposition \ref{p2}, we get
	$$
	\norm{\phi - \phi_n^M }_\infty = 
	\begin{cases}
		O\left(h^{2r + 2}\right)~~~          &  r \ge 1,\\
		O\left(h^{3}\right)~~~          & r=0.\\
	\end{cases} 
	$$
	Further, by Proposition \ref{p3} and estimate \eqref{Eq: 24}, we obtain
	$$ 
	\norm{\phi - \tilde{\phi}_n^M }_\infty = 
	\begin{cases}
		O\left(h^{2r + 3}\right)~~~          &  r \ge 1,\\
		O\left(h^{4}\right)~~~          & r=0.\\
	\end{cases} 
	$$
	This completes the proof.
\end{proof}

\section{Conclusion}
We consider collocation method and its variants for approximate solutions of a Fredholm integral equation. The kernel of the integral operator is of the type of Green's function, and the projection is chosen to be an interpolatory projection at $2r+1$ collocation points. We have seen that if we choose collocation points as Gauss points (i.e. zeros of the Legendre polynomials), then the iterative methods improve the orders of convergence. It has been observed that if the collocation points are not the Gauss points, then the collocation solution $\phi_n^C$ and iterated collocation solution $\phi_n^S$ converge at the same rate. 

For example, consider $X = C [0, 1]$ and $X_n$ be the space of piecewise constant functions with respect to the uniform partition of $[0,1]$, defined as $\{0 = t_0 < t_1 < \cdots < t_n = 1\}$, where $t_j = \frac{j}{n}$ for $j = 0,1, \ldots, n$. Let $P_n : C[0, 1] \to X_n$ be the interpolatory map at the points $ \tau_j = \frac{3j-2}{3n}$,  for $j = 1, \ldots, n$. Consider the Fredholm integral equation as
$$ 
x - Kx = 2s, \quad x \in X, \; s \in [0,1],
$$ 
and let  $K$ be the following linear operator defined on $C[0, 1]$ as
$$
(Kx)(s) =  \int_0^1 s \: x(t) \: dt  , \quad s \in [0, 1].
$$
Note that $K$ is a compact operator and $1$ is not an eigenvalue of $K$. Also, we see that the above integral equation has a unique solution in $C[0, 1]$ as  $\phi(s) = 4s$ for all $s \in [0, 1]$.  
Observe that
$$
(I - P_n) \phi(s)= 4s - \frac{4(3j-2)}{3n}, \quad s \in  [t_{j-1}, t_{j} ],
$$
and 
$$ 
\int_0^1 (I -P_n) \phi(t) ~dt =\sum_{j=1}^{n} ~ \int_{t_{j-1}}^{t_j} \left(4t -\frac{4(3j-2)}{3n} \right) ~ dt = \frac{2}{3n}.
$$
Therefore
$$
\norm{K(I -P_n) \phi }_\infty  =  \max_{s\in [0,1]} \left| \left( \int_0^1 (I -P_n) \phi(t) ~dt \right) s \right|  = \frac{2}{3n},
$$
$$ 
\norm{(I - P_n ) \phi}_\infty = \max_{1 \le j \le n} \left\{ \max_{s \in  [t_{j-1}, t_{j} ]} |(I - P_n ) \phi(s) | \right\} \le \frac{8}{3n}.
$$
We know that the iterated collocation solution is superconvergent if and only if
$$
\frac{\norm{K(I -P_n) \phi }_\infty }{ \norm{(I - P_n ) \phi}_\infty} \to 0 \; \text{as} \; n \to \infty.
$$
In this case, the above condition fails as we see that
$$
\frac{\norm{K(I -P_n) \phi }_\infty }{ \norm{(I - P_n ) \phi}_\infty} \ge \frac{1}{4}.
$$

However, in our article we consider the sequence of projections as the interpolation at $2r + 1$ points (which are not Gauss points) onto a space of piecewise polynomials of degree $\leq 2r$ with respect to a uniform partition of $[0, 1]$. Then we apply the iterated collocation and iterated modified collocation methods based on these projections. The results obtained in the Theorem \ref{t4} and Theorem \ref{t5}, in which the orders of convergence of the iterated solutions are obtained are the main contribution of this paper. These results show the improvement in the orders of convergence for both the iterated collocation solution and iterated modified collocation solution, respectively.


\section*{Declarations}
The authors have no conflict of interest to declare.



\begin{thebibliography}{99}
	
	\bibitem{KEA0}
	K. E. Atkinson, The numerical solutions of integral equations of the second kind, Cambridge University Press, Cambridge, 1997.
	
	\bibitem{boor} C. de Boor, A bound on the $L_\infty$-norm of $L_2$-approximation by splines in terms of a global mesh ratio. \textit{Mathematics of Computation} \textbf{30}:136 (1976), 765-771. \burl{https://doi.org/10.2307/2005397}
	
	\bibitem{Cha-Leb1} F. Chatelin and R. Lebbar, The iterated projection solution for the Fredholm integral equation of second kind. \textit{The ANZIAM Journal} \textbf{22}:4 (1981), 439-451. \burl{https://doi.org/10.1017/S0334270000002782}
	
	\bibitem{Cha-Leb2} F. Chatelin and R. Lebbar, Superconvergence Results for the Iterated Projection Method Applied to a Fredholm Integral Equation of the Second Kind and the Corresponding Eigenvalue Problem. \textit{Journal of Integral Equations} \textbf{6}:1 (1984), 71?91. \burl{http://www.jstor.org/stable/26164159}
	
	\bibitem{Cha1} G. A. Chandler, Superconvergence of numerical solutions to second kind integral equations. The Australian National University (Australia), 1979.
	
	\bibitem{Hild}
	F. B. Hilderbrand, Introduction to Numerical Analysis, 2nd edition, McGraw-Hill, 1974.
	
	\bibitem{RPK-1} R. P. Kulkarni, A new superconvergent collocation method for eigenvalue problems. \textit{Mathematics of Computation} \textbf{75}:254 (2006), 847-857. \burl{https://www.jstor.org/stable/4100314}
	
	\bibitem{RPK1}
	R. P. Kulkarni, A superconvergence result for solutions of compact operator equations.
	\textit{Bulletin of the Australian Mathematical Society } \textbf{68}:3 (2003), 517-528. \burl{https://doi.org/10.1017/S0004972700037916}
	
	\bibitem{RPKTJN} R. P. Kulkarni and T. J. Nidhin,  Approximate solution of Urysohn integral equations with non-smooth kernels. \textit{Journal of Integral Equations and Applications} \textbf{28}:2 (2016), 221-261. \burl{https://doi.org/10.1216/JIE-2016-28-2-221}
	
	\bibitem{RANE} 
	A. S. Rane, A note on improvement by iteration of approximate solutions of operator equations. \textit{The Journal of Analysis} \textbf{25} (2017), 177-186. \burl{https://doi.org/10.1007/s41478-017-0035-8} 
	
	\bibitem{Sl1} I. H. Sloan, Improvement by iteration for compact operator equations. \textit{Mathematics of Computation} \textbf{30}:136 (1976), 758-764. \burl{https://doi.org/10.2307/2005396}
	
	
	\bibitem{Qun} Q. Lin, S. Zhang and N. Yan, An acceleration method for integral equations by using interpolation post-processing. \textit{Advances in Computational Mathematics} \textbf{9} (1998), 117-129. \burl{https://doi.org/10.1023/A:1018925103993}
	
\end{thebibliography}
\end{document}